\documentclass[12pt, reqno]{amsart}

\usepackage{amsmath}
\usepackage{amsthm}
\usepackage{amssymb}
\usepackage{amsrefs}
\usepackage{mathrsfs}

\usepackage[usenames]{color}

 \newtheorem{thm}{}[section]
 \newtheorem{theorem}[thm]{Theorem}
 \newtheorem{lemma}[thm]{Lemma}
 \newtheorem{proposition}[thm]{Proposition}
 \theoremstyle{remark}
 \newtheorem{remark}[thm]{Remark}

 \numberwithin{equation}{section}
 \allowdisplaybreaks

\newcommand{\NN}{\ensuremath{\mathbb{N}}}
\newcommand{\ZZ}{\ensuremath{\mathbb{Z}}}
\newcommand{\RR}{\ensuremath{\mathbb{R}}}
\newcommand{\TT}{\ensuremath{\mathbb{T}}}
\newcommand{\XX}{\ensuremath{\mathbb{X}}}
\newcommand{\CC}{\ensuremath{\mathbb{C}}}
\newcommand{\Cont}{\ensuremath{\mathcal{C}}}
\newcommand{\PP}{\ensuremath{\mathcal{P}}}
\newcommand{\xx}{\ensuremath{\mathbf{x}}}
\newcommand{\BB}{\mathcal{B}}
\newcommand{\GG}{\mathcal{G}}
\newcommand{\LL}{\mathscr{L}}

\newcommand{\supp}{\operatorname{supp}}
\newcommand{\spn}{\operatorname{span}}

\begin{document}

\title[Quasi-greedy bases  and Jacobi polynomials in $L_p$ spaces]{Unconditional and quasi-greedy bases in $L_p$ with applications to Jacobi polynomials Fourier series}

\author[F. Albiac]{Fernando Albiac}
\address{Mathematics Department\\ 
Universidad P\'ublica de Navarra\\
Campus de Arrosad\'{i}a\\
Pamplona\\ 
31006 Spain}
\email{fernando.albiac@unavarra.es}

\author[J. L. Ansorena]{Jos\'e L. Ansorena}
\address{Department of Mathematics and Computer Sciences\\
Universidad de La Rioja\\ 
Logro\~no\\
26004 Spain}
\email{joseluis.ansorena@unirioja.es}

\author[\'O. Ciaurri]{\'Oscar Ciaurri}
\address{Department of Mathematics and Computer Sciences\\
Universidad de La Rioja\\ 
Logro\~no\\
26004 Spain}
\email{oscar.ciaurri@unirioja.es}

\author[J. L. Varona]{Juan L. Varona}
\address{Department of Mathematics and Computer Sciences\\
Universidad de La Rioja\\ 
Logro\~no\\
26004 Spain}
\email{jvarona@unirioja.es}



\begin{abstract} 
We  show that  the decreasing rearrangement of the Fourier series with respect to  the Jacobi polynomials for functions in $L_p$  does not converge unless $p=2$. As a by-product of our work on quasi-greedy bases in $L_{p}(\mu)$, we show that no normalized unconditional basis in $L_p$, $p\not=2$, can be semi-normalized in $L_q$ for $q\not=p$, thus extending a classical theorem of Kadets and Pe{\l}czy{\'n}ski from 1968.
\end{abstract}


\maketitle

\section{Introduction and background}
\label{Introduction}

\noindent A \textit{biorthogonal system} for an infinite-dimensional  (real or complex) separable Banach space
$(\XX, \Vert\cdot\Vert)$
is a family $(\xx_{j},\xx_{j}^{\ast})_{j\in J}\subset \XX\times \XX^{\ast}$ verifying
\begin{itemize}
\item[(i)] $\XX= \overline{\spn \{ \xx_{j} : j\in J\}}$, and
\item[(ii)] $\xx_{j}^{\ast}(\xx_{k})=1$ if $j=k$ and $\xx_{j}^{\ast}(\xx_{k})=0$ otherwise.
\end{itemize}
For brevity, we refer to $\BB=(\xx_j)_{j\in J}$ as a \textit{basis} and to the unequivocally determined $\BB^*=(\xx_j^*)_{j\in J}$ as its
\textit{orthogonal family.}
If the biorthogonal system fulfills the additional condition
\begin{itemize}
\item[(iii)] $\sup_{j\in J} \Vert \xx_j\Vert \Vert \xx^*_j\Vert <\infty$
\end{itemize}
we say that the basis and the biorthogonal system are \textit{bounded}. Finally, when the basis verifies
\begin{itemize}
\item[(iv)] $0<\inf_{j\in J} \Vert \xx_j\Vert \le \sup_{j\in J} \Vert \xx_j\Vert <\infty$
\end{itemize}
we say that the basis 
is \textit{semi-normalized} (respectively \textit{normalized} if $\Vert \xx_j\Vert=1$ for all $j\in J$). 
Notice that a biorthogonal system fulfills simultaneously (iii) and (iv) if and only if
\[
\sup_{j \in J} \max\{ \Vert \xx_j\Vert ,\Vert \xx_j^*\Vert \}<\infty.
\]

Suppose $\BB=(\xx_j)_{j\in J}$ is a semi-normalized bounded basis in a Banach space $\XX$ with orthogonal family
$\BB^*=(\xx_j^*)_{j\in J}$.
 Each $f\in \XX$ has a unique formal series expansion in terms of the basis,
\begin{equation}
\label{formula:SeriesExpansion}
f=\sum_{j\in J} \xx_j^*(f)\xx_j.
\end{equation}
In order to try to make sense of the infinite sum in \eqref{formula:SeriesExpansion}, one can fix a bijective mapping $\pi\colon\NN\to J$ and study the convergence of the formal series
$\sum_{n=1}^\infty \xx_{\pi(n)}^*(f)\xx_{\pi(n)}$. If this series converges to $f$ for every $f\in\XX$ then $\BB$ is a \textit{Schauder basis} for the bijection~$\pi$. Schauder  bases are very well-known and have been widely studied. They are characterized as those bases for which the partial sum operators $S_{\pi,m}\colon\XX\to\XX$, given by
\begin{equation}
\label{definition:PartialSum}
 f\mapsto S_{\pi,m}(f)= \sum_{n=1}^m \xx_{\pi(n)}^*(f)\xx_{\pi(n)}
\end{equation}
are uniformly bounded. The property that  $\sum_{n=1}^\infty \xx_{\pi(n)}^*(f)\xx_{\pi(n)}$ converges for any $f\in\XX$ and any bijection $\pi$
 yields the more restrictive
class of \textit{unconditional bases}. Recall that, equivalently,  a basis is unconditional if and only if
for every choice of signs $\varepsilon=(\varepsilon_j)_{j\in J}\in\{-1,1\}^J$ the multiplier
\[
P_\varepsilon\colon \XX\to \XX, \quad f\mapsto \sum_{j\in J} \varepsilon_j\xx_j^*(f) \xx_j
\]
is well defined and the family of operators $(P_\varepsilon)_{\varepsilon\in \{\pm 1\}^J}$ is uniformly bounded.

An \textit{ordering} for an element $f\in\XX$ (with respect to a basis $\BB$) is a one-to-one map $\rho\colon\NN\to J$ such that $\supp(f) := \{j\in J : \xx_j^*(f)\not=0\} \subseteq \rho(\NN)$.
From the point of view of  approximation theory, given a function $f$ in $\XX$ and an ordering $\rho$ for $f$, the sequence $(S_{\rho,m}(f))_{m=1}^\infty$
constructed as in \eqref{definition:PartialSum} defines an algorithm to approximate to~$f$. The minimal requirement we must impose to $\rho$ is that  $(S_{\rho,m}(f))_{m=1}^\infty$ converges to~$f$. In case $\BB$ is a Schauder basis for some bijection $\pi$, the algorithm based on $\pi$ fulfills this requirement for any $f\in\XX$. The independence of the ordering from the vector determines both the goodness and the limitations of this approximation algorithm for Schauder bases. The operators $S_{\pi,m}$ are linear and uniformly bounded,
but it is natural to wonder if by allowing the ordering to depend on each particular vector we can attain a higher rate of convergence.

The most important algorithm based on letting the ordering depend on the vector is the \textit{greedy algorithm}, also known as the
\textit{thresholding algorithm}. Since for each $f\in\XX$ the sequence $(\xx^*_j(f))_{j\in J}$ belongs to $c_0(J)$, there is an ordering $\rho$ for $f$ such that 
\begin{equation}
\label{greedycondition}
|\xx_{\rho(k)}^{\ast}(f)|\ge |\xx_{\rho(n)}^{\ast}(f)| \quad\text{if}\quad k\le n.
\end{equation}
If the family $(\xx_j^*(f))_{j\in J}$ contains several terms with the same absolute value then such an ordering for $f$ is not uniquely determined. 
In order to get uniqueness, we fix a ``natural'' bijection $\tau \colon J \to \NN$, and we impose the additional condition
\begin{equation}
\label{desempate}
\tau(\rho(k))\le\tau(\rho(n)) \quad\text{whenever}\quad |\xx_{\rho(k)}^{\ast}(f)|= |\xx_{\rho(n)}^{\ast}(f)|.
\end{equation}
If $f$ is infinitely supported, there is a unique ordering $\rho$ for $f$ which fulfills \eqref{greedycondition} and \eqref{desempate}, and such an ordering verifies $\rho(\NN)=\supp (f)$. In the case in which $f$ is finitely supported, there is a unique ordering $\rho$ for $f$ which fulfills \eqref{greedycondition}, \eqref{desempate} and the extra property $\rho(\NN)=J$. In any case, we will refer to such a unique ordering as the \textit{greedy ordering} for~$f$.
For each $m\in \NN$, the \textit{$m$-term greedy approximation} to $f$
is given by
\[
  \GG_{m}[\BB, \XX](f):=\GG_{m}(f) =S_{\rho,m}(f)= \sum_{n=1}^{m}\xx_{\rho(n)}^{\ast}(f)\xx_{\rho(n)},
\]
where $\rho$ is the greedy ordering for $f$, and the sequence $(\GG_m(f))_{m=1}^\infty$ is called the greedy algorithm for $f$ with respect to the basis~$\BB$.

Konyagin and Temlyakov \cite{KonyaginTemlyakov1999} defined a basis to be \textit{quasi-greedy} if $\lim_{m\to \infty}\GG_{m}(f)=f$ for $f\in \XX$, that is, the
greedy algorithm with respect to the basis $\BB$ converges in the Banach space~$\XX$. Subsequently, Wojtaszczyk \cite{Wo2000} proved that these are precisely the bases for which the greedy operators $(\GG_m)_{m=1}^{\infty}$ are uniformly bounded i.e., there exists a constant $C\ge 1$ such that, for all $f\in \XX$ and $m\in \NN$,
\begin{equation}
\label{result:WoCondition}
\Vert \GG_m(f)\Vert\le C \Vert f\Vert.
\end{equation}
Notice the similarity between \eqref{result:WoCondition} and the  characterization of Schauder bases.  However,  the operators $\GG_m$ are neither linear nor continuous. 
We emphasize that, as Wojtaszczyk pointed out in \cite{Wo2000}, the choice of the bijection $\tau$ with respect to which we construct the greedy algorithm $(\GG_{m})_{m=1}^{\infty}$ plays no relevant role in the theory.

Unconditional bases are a special kind of quasi-greedy bases. Although the converse is not true in general, quasi-greedy bases always retain in a certain sense a flavor of unconditionality. For example, they are \textit{unconditional for constant coefficients} \cite{Wo2000}, i.e., there is a constant $C$ (to be precise $C=2 C_{w}$, where $C_{w}$ is the least constant in \eqref{result:WoCondition}, works) such that
\begin{equation}
\label{result:UCCCondition}
 \frac{1}{C}\left\Vert \sum_{j\in A} \xx_{i}\right\Vert\le \left\Vert \sum_{j\in A}\varepsilon_{i} \xx_{i}\right\Vert\le C \left\Vert \sum_{j\in A} \xx_{i}\right\Vert
\end{equation}
for any finite subset $A$ of $J$ and any choice of signs $\varepsilon_j\in\{\pm1\}$.

Before the concept of quasi-greedy basis was introduced in the literature, C\'ordoba and Fern\'andez \cite{CoFe1998} had studied the convergence of decreasing rearranged Fourier series. For $k\in\ZZ$ let us define $\tau_k\colon\RR\to\CC$ by $\tau_k(x)=e^{2\pi k x i}$. Let $1\le p<\infty$ and denote by $q$ its conjugate exponent, determined by
$1/p+1/q=1$. Then, with the usual identification of  $L_p^*(\TT)$ with $L_q(\TT)$, the double sequence $(\tau_k,\tau_{-k})_{k=-\infty}^\infty$ is a normalized bounded biorthogonal system for $L_p(\TT)$. The authors of \cite{CoFe1998} showed that for each $1<p<2$ there is a function $f\in L_p(\TT)$ whose decreasing rearranged Fourier series does not converge, which in our language can be stated as saying that  the trigonometric system $(\tau_k)_{k=-\infty}^\infty$ is not a quasi-greedy basis for $L_p(\TT)$.
Combining the condition characterizing quasi-greedy bases \eqref{result:WoCondition} with \cite{Telmyakov1998}*{Remark~2}, the result extends to the whole range of $p\in[1,\infty]\setminus\{2\}$ (replacing $L_p(\TT)$ with $\Cont(\TT)$ when $p=\infty$).
Wojtaszczyk gave a different proof of this result in \cite{Wo2000} that relies on~\eqref{result:UCCCondition}.

A natural way to continue this line of research is to consider Fourier series with respect to orthonormal bases. Let $(X,\Sigma,\mu)$ be a measure space such that the Hilbert space $L_2(\mu)$ is separable. Let
$(\xx_j)_{j\in J}$ be an orthonormal basis of $L_2(\mu)$.
For $1\le p< \infty$, let $q$ be its conjugate exponent. 
In case that  $\spn \{ \xx_j : j\in J\}$ is dense in $L_p(\mu)$ and
\[
\sup_{j\in J} \Vert \xx_j \Vert_p \Vert \xx_j \Vert_q<\infty,
\]
the identification of $L_p^*(\mu)$ with $L_q(\mu)$, yields that $(\xx_j ,\overline{\xx_j})_{j\in J}$ is a bounded biorthogonal system for $L_p(\mu)$. It therefore makes sense to investigate the convergence of the greedy algorithm with respect to the $L_p(\mu)$-normalized system
\[
\left( \Vert \xx_j \Vert_p^{-1} {\xx_j}, \Vert \xx_j\Vert_p \overline{ \xx_j } \right)_{j\in J}.
\]
Notice that if the measure $\mu$ is finite and the orthonormal basis $(\xx_j)_{j\in J}$
 is \textit{uniformly bounded}, i.e.\ $\sup_{j\in J} \Vert \xx_j\Vert_\infty<\infty$, then it is semi-normalized and bounded in $L_p(\mu)$ for any $1\le p<\infty$.
Nielsen \cite{Nielsen2007} proved that there is an uniformly bounded orthonormal basis of $L_2(\TT)$ which is quasi-greedy for $L_p(\TT)$ for any $1<p<\infty$, thus
exhibiting a behavior opposite to that of the trigonometric system.

In this paper we focus on Jacobi polynomials.
Recall that, for  scalars $\alpha$, $\beta>-1$, the  $L_2(\mu_{\alpha,\beta})$-normalized Jacobi polynomials
$(p_n^{(\alpha,\beta)})_{n=0}^\infty$ appear as the orthonormal polynomials associated to the measure $\mu_{\alpha,\beta}$ given by
\begin{equation}
\label{measureJacobi}
d\mu_{\alpha,\beta}(x)=(1-x)^\alpha (1+x)^\beta \chi_{(-1,1)} (x) \, dx.
\end{equation}
Since polynomials are dense in $L_p(\mu_{\alpha,\beta})$ for any $1\le p<\infty$, Jacobi polynomials of indices $\alpha$ and $\beta$ constitute an orthonormal basis of $L_2(\mu_{\alpha,\beta})$. Our main result on Jacobi polynomials establishes that the greedy algorithm for this kind of orthogonal polynomials follows the same pattern as the greedy algorithm for the trigonometric system.

\begin{theorem}
\label{result:MainTheorem}
Let $1\le p<\infty$ and $\min\{\alpha,\beta\}> -1/2$. The $L_p(\mu_{\alpha,\beta})$-normalized Jacobi polynomials of indices $\alpha$ and $\beta$ form a quasi-greedy basis for $L_p(\mu_{\alpha,\beta})$ if and only if $p=2$.
\end{theorem}

Section~\ref{Jacobi} is devoted to prove Theorem~\ref{result:MainTheorem}.
Before, in Section~\ref{UnconditionalLp} we develop
the functional analysis machinery that we will need in order to do that and 
we show the following result on unconditional bases in $L_p(\mu)$-spaces.

\begin{theorem}
\label{NotSimultaneouslyBounded}Let $\mu$ be a finite measure and $p\in(1,\infty)\setminus\{2\}$. Suppose that $(\xx_j)_{j\in J}$ is a semi-normalized unconditional basis of a non-Hilbertian Banach space $\XX\subseteq L_p(\mu)$. Suppose also that $\XX$ is complemented in $L_p(\mu)$. Then
\begin{itemize}
\item[(i)] $\limsup_{j\in J} \Vert \xx_j\Vert_q=\infty$ for any $p<q$, and
\item[(ii)] $\liminf_{j\in J} \Vert \xx_j\Vert_q=0$ whenever $\max\{q,2\}<p$.
\end{itemize}
\end{theorem}
Notice that Theorem~\ref{NotSimultaneouslyBounded} is relevant for its intrinsic importance within the framework of the theory of bases. Firstly, it extends to any $q$ a result that Kadets and Pe{\l}czy{\'n}ski proved only for $q=2$ (see \cite{KadPel1962}*{Corollary 9}). Secondly, it  generalizes the main result of Gapo\v{s}kin in \cite{Gaposkin1958}, where he shows that no normalized unconditional basis in
$L_p[0,1]$ can be uniformly bounded. Lastly, for finite measures, Theorem~\ref{NotSimultaneouslyBounded} overrides a recent result of the first two authors that says that if $\mu$ is a nonpurely atomic measure then there is no basis $\BB$ that is simultaneously greedy (see the definition below) in two different $L_{p}(\mu)$ spaces, $1<p<\infty$ \cite{AA2015}*{Theorem~4.4}.

We end this preliminary section by singling out some notation and terminology that will be used heavily throughout.
Given families of positive real numbers $(\alpha_i)_{i\in I}$ and $(\beta_i)_{i\in I}$, the symbol $\alpha_i\lesssim \beta_i$ for $i\in I$ means that
$\sup_{i\in I}\alpha_i/\beta_i <\infty$, while $\alpha_i\approx \beta_i$ for $i\in I$ means that $\alpha_i\lesssim \beta_i$ and $\beta_i\lesssim \alpha_i$ for $i\in I$.

A basis $\BB=(\xx_j)_{j\in J}$ in a Banach space $\XX$ is said to be \textit{democratic} if
there is a constant $D\ge 1$ such that
\[
\left\Vert \sum_{j\in A}\xx_j \right\Vert \le D \left\Vert \sum_{j\in B} \xx_j \right\Vert
\]
whenever $A$ and $B$ are finite subsets of $J$ with $|A|=|B|$.
To quantify the democracy of a basis $\BB$ we consider the \textit{upper democracy function} of~$\BB$ (also known as the \textit{fundamental function of~$\BB$}) given by
\[
 \varphi_{u}[\BB,\XX](N)=\sup_{|A|\le N}\left\Vert \sum_{j\in A}\xx_j\right\Vert, \qquad N\in\NN,
\]
and the \textit{lower democracy function} of $\mathcal B$ in $\XX$, defined as
\[
 \varphi_{l}[\BB,\XX](N)=\inf_{|A|\ge N}\left\Vert \sum_{j\in A}\xx_j\right\Vert, \qquad N\in\NN.
\]
A quasi-greedy basis $\mathcal B$ is democratic if and only if $\varphi_{u}[\BB,\XX](N)\approx \varphi_{l}[\BB,\XX](N)$ for $N\in\NN$.

A basis $(\xx_n)_{n=1}^\infty$ is said to be \textit{almost greedy} if there is a constant $C\ge 1$ such that
\[
 \Vert x-\GG_m(x)\Vert \le C \inf\left\{\left \Vert x-\sum_{j\in A} \xx_j^*(x) \xx_j\right\Vert : |A|=m\right\}
\]
for all $m\in \NN$ and $x\in \XX$. 
Dilworth et al.\ \cite{DKKT2003} characterized almost greedy basis as those bases that are simultaneously quasi-greedy and democratic.

Finally, the best one can hope for in regards to the greedy algorithm is the existence of a constant $C\ge 1$ such that
\[
 \Vert x-\GG_m(x)\Vert \le C \inf\left\{\left \Vert x-\sum_{j\in A} a_{j} \xx_j\right\Vert : |A|=m,\; (a_{j})_{j\in A}\;\;\text{scalars}\right\}
 \]
for all $m\in \NN$ and $x\in \XX$. If this is the case, the basis is called \textit{greedy}. Konyagin and Temlyakov \cite{KonyaginTemlyakov1999} characterized greedy bases as those bases that are unconditional and democratic.

If necessary, the reader will find more background on Banach space theory in \cite{AlbiacKalton2006} 
and on orthogonal polynomials  in~\cite{Szego}. 

\section{Quasi-greedy and unconditional bases in $L_p(\mu)$-spaces}
\label{UnconditionalLp}

\noindent We start generalizing to quasi-greedy bases a fact which is standard for unconditional bases in $L_p(\mu)$-spaces.

\begin{lemma}
\label{result:Random}
Let $(X,\Sigma,\mu)$ be a finite measure space. Let $1\le p<\infty$ and $(\xx_j)_{j\in J}$ be a quasi-greedy basis for
a separable subspace of $L_p(\mu)$. Then for $A\subseteq J$ finite,
\[
\left\Vert \sum_{j\in A}\xx_j \right\Vert_p
\approx \left\Vert \left(\sum_{j\in A} |\xx_j|^2\right)^{1/2}\right\Vert_p.
\]
\end{lemma}

\begin{proof}
Let $(\varepsilon_j)_{j\in J}$ be a Rademacher family defined on some probability space $(\Omega,P)$, and $A\subseteq J$ finite. Combining \eqref{result:UCCCondition}, Fubini's
theorem, and Khintchine's inequality yields 
\begin{align*}
\left\Vert \sum_{j\in A}\xx_j \right\Vert_p
&\approx \left( \int_\Omega \Big\Vert \sum_{j\in A}\varepsilon_j \xx_j \Big\Vert_p^p\, dP\right)^{1/p}\\
&= \left(\int_X \int_\Omega \Big| \sum_{j\in A} \varepsilon_j \xx_j\Big|^p \, dP \,d\mu\right)^{1/p}\\
&\approx \left\Vert \left(\sum_{j\in A} |\xx_j|^2\right)^{1/2}\right\Vert_p.
\qedhere
\end{align*}
\end{proof}

Our next auxiliary result displays an estimate that is implied when a family of functions is simultaneously seminormalized in two different $L_{p}$ spaces.

\begin{lemma}
\label{result:PreviousDemocracyEstimates} 
Let $1\le p< q\le 2$ (respectively, $2\le q <p\le\infty$) and let $(f_j)_{j\in J}$ be a family of measurable functions defined on a finite measure space $(X,\Sigma,\mu)$. Suppose that
$\Vert f_j\Vert_p\approx \Vert f_j \Vert_q\approx 1$ for $j\in J$. Then, for $A\subseteq J$ finite,
\[
|A|^{1/2} \lesssim \left\Vert \left( \sum_{j\in A}| f_j|^2\right)^{1/2}\right\Vert_p \lesssim |A|^{1/q}
\]
(respectively,
\[
|A|^{1/q} \lesssim \left\Vert \left( \sum_{j\in A}| f_j|^2\right)^{1/2}\right\Vert_p \lesssim |A|^{1/2}).
\]
\end{lemma}

\begin{proof}
Assume $1\le p<q\le 2$. Using the embeddings $\ell_q\subseteq\ell_2$ and $L_q(\mu)\subseteq L_p(\mu)$,
\begin{equation*}
\left\Vert \left(\sum_{j\in A} | f_j|^2\right)^{1/2}\right\Vert_p
\lesssim \left\Vert \left(\sum_{j\in A} | f_j|^q\right)^{1/q}\right\Vert_q
=\left(\sum_{j\in A} \Vert f_j\Vert_q^q\right)^{1/q}
\approx |A|^{1/q}.
\end{equation*}
Let $r=p/2<1$. Using that $\Vert f+g\Vert_r \ge \Vert f\Vert_r +\Vert g\Vert_r$ whenever $f$ and $g$ are measurable positive functions,
\begin{align*}
 \left\Vert \left(\sum_{j\in A} | f_j|^2\right)^{1/2}\right\Vert_p
&= \left\Vert \sum_{j\in A} | f_j|^2\right\Vert_r^{1/2}\\
&\ge \left( \sum_{j\in A} \Vert \,  |f_j|^2 \, \Vert_r \right)^{1/2}\\
&=\left( \sum_{j\in A} \Vert f_j \Vert_p^2 \right)^{1/2}\\
&\approx |A|^{1/2}.
\end{align*}

The case $2\le q<p\le\infty$ follows from a ``dual'' argument.
\end{proof}

\begin{lemma}
\label{result:DemocracyEstimates}
Let $(X,\Sigma,\mu)$ be a finite measure space. Suppose $1\le p< q\le 2$ (respectively, $2\le q <p<\infty$). Let  $(\xx_j)_{j\in J}$ be a quasi-greedy basis for
a separable subspace $\XX$ of $L_p(\mu)$ such that$\Vert \xx_j \Vert_q\approx 1$ for $j\in J$. Then, for $N\in\NN$,
\[
N^{1/2} \lesssim \varphi_{l}[\BB,\XX](N) \le \varphi_{u}[\BB,\XX](N) \lesssim N^{1/q}
\]
(respectively,
\[
N^{1/q} \lesssim \varphi_{l}[\BB,\XX](N) \le \varphi_{u}[\BB,\XX](N) \lesssim N^{1/2}).
\]
\end{lemma}

\begin{proof} 
Quasi-greedy bases are semi-normalized, so $\Vert \xx_j \Vert_p\approx 1$ for $j \in J$. 
Then, just put together Lemma~\ref{result:Random} and Lemma~\ref{result:PreviousDemocracyEstimates}.
\end{proof}

The next two propositions are on-the-spot corollaries of Lemma~\ref{result:PreviousDemocracyEstimates} and Lemma~\ref{result:DemocracyEstimates}, respectively. We point out that a similar statement to Proposition~\ref{result:Democracy} with the stronger assumption that the basis be uniformly bounded was obtained by Dilworth et al.\ \cite{DST2012}*{Proposition~2.17}.

\begin{proposition}
\label{result:LdosSumEstimate} 
Let $1\le p\le \infty$. Suppose $(f_j)_{j\in J}$ is a family of measurable functions defined on a finite measure space $(X,\Sigma,\mu)$ such that
$\Vert f_j\Vert_p\approx \Vert f_j \Vert_2\approx 1$ for $j\in J$. Then for $A\subseteq J$ finite,
\[
 \left\Vert \left( \sum_{j\in A}| f_j|^2\right)^{1/2}\right\Vert_p \approx |A|^{1/2}.
\]
\end{proposition}

\begin{proposition}
\label{result:Democracy}
Let $(X,\Sigma,\mu)$  be a finite measure space and let $1\le p<\infty$. Suppose  $\BB=(\xx_j)_{j\in J}$ is a quasi-greedy basis for
a separable subspace $\XX$ of $L_p(\mu)$ with $\Vert \xx_j \Vert_2\approx 1$ for $j\in J$. 
Then $\BB$ is democratic, hence almost greedy, and
its democracy functions verify 
\[
\varphi_l[\BB,\XX](N)\approx N^{1/2}\approx \varphi_u[\BB,\XX](N),\quad N\in\NN.
\]
\end{proposition}

We are now en route to completing the proof of Theorem~\ref{NotSimultaneouslyBounded}. Before we do so, we write down two classical results in the isomorphic theory of Banach spaces which are very well-known to the specialists. In order to make the paper as self-contained as possible we sketch their proofs.

\begin{theorem}
\label{result:KaPeSubbasis} 
Let $(X,\Sigma,\mu)$ be a measure space. Suppose that $\BB$ is a seminormalized unconditional basis of a non-Hilbertian Banach space $\XX\subseteq L_p(\mu)$, $1<p<\infty$. Suppose also that $\XX$ is complemented in $L_p(\mu)$. Then
$\BB$ has a subbasis equivalent to the unit vector basis of~$\ell_p$.
\end{theorem}

\begin{proof}
 Without loss of generality we may and do assume that $L_p(\mu)$ is separable. Then $L_p(\mu)$ is isomorphic either to $L_p[0,1]$ or to~$\ell_p$
(see~\cite{JohnsonLindes2001}).
In the first case, the same argument used by Kadec and Pe{\l}czy{\'n}ski to prove \cite{KadPel1962}*{Theorem~4} leads to our goal. In the last case, by the Bessaga-Pe{\l}czy{\'n}ski selection principle (\cite{BePe1958}*{p.~214}), $\BB$ has a subbasis equivalent to a block basic sequence of the unit vector basis of~$\ell_p$.
Since the unit vector basis of $\ell_p$ is perfectly homogeneous, this subbasis is equivalent to the unit vector basis of~$\ell_p$.
\end{proof}

\begin{lemma}
\label{result:EquivalentNorms} 
Let $(X,\Sigma,\mu)$ be a finite measure space and let $0<p<q\le \infty$. Consider a subset $M\subseteq L_p(\mu)$ such that $\Vert f\Vert_p\approx\Vert f\Vert_q$ for $f\in M$. Then, for any $0<r\le q$, $\Vert f\Vert_r \approx\Vert f\Vert_p\approx\Vert f\Vert_q$ for $f\in M$.
\end{lemma}

\begin{proof} 
The result is obvious for $p\le r\le q$, so we assume that $r<p$. Then, it is also obvious that $\Vert f\Vert_r\lesssim \Vert f\Vert_p$ for $f\in M$. To prove the reverse inequality, consider $0<a<p$ such that $a/q+(p-a)/r=1$.
Let $f\in M$. By H\"older's inequality,
\begin{align*}
\Vert f\Vert_p^p 
&=\int_X |f|^a |f|^{p-a} \, d \mu\\
&\le \left( \int_X |f|^{q}\, d\mu \right)^{a/q} \left( \int_X |f|^{r}\, d\mu \right)^{(p-a)/r}\\
&=\Vert f\Vert_q^a \Vert f\Vert_r^{p-a}\\
&\lesssim \Vert f\Vert_p^a \Vert f\Vert_r^{p-a}.
\end{align*}
Simplifying, we get $\Vert f\Vert_p^{p-a}\lesssim \Vert f\Vert_r^{p-a}$.
\end{proof}

\begin{proof}[Proof of Theorem~\ref{NotSimultaneouslyBounded}] 
Let $(\xx_j)_{j\in J}$ be a semi-normalized unconditional basis of a non-Hilbertian Banach space $\XX\subseteq L_p(\mu)$, where $p\in(1,\infty)\setminus\{2\}$ and $\mu$ is finite.
We divide the proof in three cases.

\smallskip

\noindent\textsc{Case 1:} $1<p<2$ and $p<q$. Assume that $\limsup_j \Vert x_j\Vert_q<\infty$.
We can suppose, without loss of generality, that $p<q\le 2$.
Let $J_0=\{ j : \Vert \xx_j\Vert_q<\infty \}$, $\BB_0=(\xx_j)_{j\in J_0}$ and then define
$\XX_0$ as the closed subspace spanned by $\BB_0$ in $L_p(\mu)$.
We have that $J\setminus J_0$ is finite and that
$\Vert \xx_j\Vert_q\approx 1$ for $j\in J_0$. By Lemma~\ref{result:DemocracyEstimates}, $\varphi_u[\BB_0,\XX_0](N)\lesssim N^{1/q}$. Furthermore, by
Theorem~\ref{result:KaPeSubbasis}, $\BB_0$ has a subbasis equivalent to the unit vector basis of $\ell_p$. Therefore,
$N^{1/p}\lesssim \varphi_u[\BB_0,\XX_0](N)$. Combining, we obtain $N^{1/p}\lesssim N^{1/q}$. This absurdity proves the result. 

\smallskip

\noindent\textsc{Case 2:} $2<p<\infty$ and $q<p$. This case is the ``dual'' of the previous one. Since  its proof is similar we leave it out to the reader.

\smallskip

\noindent\textsc{Case 3:} $2<p<q$. Suppose that $\limsup_j \Vert x_j\Vert_q<\infty$. Removing a finite set of terms from $\BB$ we get
$\Vert \xx_j\Vert_q\approx 1$. By Lemma~\ref{result:EquivalentNorms}, $\Vert \xx_j\Vert_2 \approx 1$, contradicting the already proven Case~2.
\end{proof}

\begin{remark} 
The proof of Theorem~\ref{NotSimultaneouslyBounded}
reinforces the role of  democracy as a hinge property in the study of unconditional bases in Banach spaces. This idea was already inferred in the work of Zippin~\cite{Zippin1966}, where he characterizes perfectly homogeneous bases.
\end{remark}

We close with the analogous result to Theorem~\ref{NotSimultaneouslyBounded} for the  case $p=1$.
To better understand the statement and its proof we recall that, given $1\le p\le \infty$, an infinite-dimensional Banach space $\XX$ is said to be a \textit{$\LL_p$-space} if there is $\lambda\ge 1$ such that for every finite-dimensional subspace $E\subseteq \XX$ there is $d\in\NN$ and a $d$-dimensional subspace
$E\subseteq F\subseteq \XX$ that satisfies $d(F,\ell_p^d)\le \lambda$. For $1<p<\infty$, $\LL_p$-spaces are characterized as non-Hilbertian complemented subspaces of $L_p(\mu)$-spaces, while $\XX$ is an $\LL_1$-space if and only if
$\XX^{**}$ is isomorphic to a complemented subspace of an $L_1(\mu)$-space. Hence, since $\XX$ embeds isometrically in $\XX^{**}$, it is natural to regard 
$\LL_1$-spaces as (possibly non complemented) subspaces of $L_1(\mu)$-spaces.
A fundamental property is that any $\LL_1$-space has the \textit{Grothendieck's Theorem property} (is a GT-space, for short).
 We refer to \cites{LinPel1968, LinRos1969} for details.
 
\begin{proposition}
\label{result:NotBoundedOne}
Let $\mu$ be a finite measure. Suppose that $\BB=(\xx_j)_{j\in J}$ is a quasi-greedy basis for a Banach space
$\XX\subseteq L_1(\mu)$.
Assume also that $\XX$ is a GT-space.
Then $\limsup_{j\in J} \Vert \xx_j\Vert_q=\infty$ for any $1< q\le\infty$.
\end{proposition}

\begin{proof}
Appealing to \cite{DST2012}*{Theorem~4.2} we get
$\varphi_u[\BB,\XX](N)\approx N$.
Assume that $\limsup_{j\in J} \Vert \xx_j\Vert_q<\infty$ for some $1<q\le\infty$.
Then, without lost of generality we can suppose that $q\le 2$ and, by removing a finite set of terms from $\BB$ if necessary, that $\Vert \xx_j \Vert_q\approx 1$. Applying Lemma~\ref{result:DemocracyEstimates} we get $\varphi_u[\BB,\XX](N)\lesssim N^{1/q}$, which leads to $N\lesssim N^{1/q}$, a contradiction.
\end{proof}



\begin{remark} 
The subspace spanned by the Rademacher functions in $L_{p}(\mu)$ serves as an example to show that the assumption
``non-Hilbertian'' cannot be dropped from
Theorem~\ref{NotSimultaneouslyBounded} and that the assumption of 
$\XX$ being a GT-space cannot be dropped from
Proposition~\ref{result:NotBoundedOne}.
\end{remark}

\section{The greedy algorithm for Jacobi polynomials}
\label{Jacobi}

\noindent In this section, besides the orthonormal polynomials $(p_n^{(\alpha,\beta)})_{n=0}^\infty$ defined in Section~\ref{Introduction}, we consider
 the polynomials $(P_n^{(\alpha,\beta)})_{n=0}^\infty$ which are orthogonal for the measure 
$\mu_{\alpha,\beta}$ defined in \eqref{measureJacobi} and verify the \textit{normalization condition}
\begin{equation}
\label{ValueInOne}
P_n^{(\alpha,\beta)} (1)=\binom{n+\alpha}{n}, \quad n\in \NN\cup\{0\}.
\end{equation}
Of course, there are positive scalars $d_n$ such that $p_n^{(\alpha,\beta)}=d_n P_n^{(\alpha,\beta)}$, $n\ge 0$. It is well known that the normalization sequence 
$(d_n)_{n=0}^\infty $ verifies
\begin{equation}
\label{eq:asym-dn}
d_n
= \Big(\frac {(2n+\alpha+\beta+1)n!\,\Gamma(n+\alpha+\beta+1)}
{2^{\alpha+\beta+1}\Gamma(n+\alpha+1)\Gamma(n+\beta+1)}
 \Big)^{1/2}
\approx n^{1/2}, \, n\in\NN.
\end{equation}
In what follows, with the aim of avoid cumbrous notations, we will denote by $(n^{1/2} P_n^{(\alpha,\beta)})_{n=0}^\infty$ the extension of $(n^{1/2} P_n^{(\alpha,\beta)})_{n=1}^\infty$ whose $0$-term
is the constant function~$1$.
In the light of \eqref{eq:asym-dn},
it is reasonable to expect that the sequences $(p_n^{(\alpha,\beta)})_{n=0}^\infty$ and $(n^{1/2} P_n^{(\alpha,\beta)})_{n=0}^\infty$ behave similarly.
 Wojtaszczyk \cite{Wo2000} confirmed this fact by showing that quasi-greedy bases verify the following perturbation principle. 

\begin{theorem}[cf.~\cite{Wo2000}*{Proposition~3}]
\label{result:Perturbing} 
Suppose that $(\xx_j)_{j\in J}$ is a quasi-greedy basis for a Banach space~$\XX$. Let $(\lambda_j)_{j\in J}$ be a family of scalars such that
\[
0<\inf_{j\in J} |\lambda_j|\le \sup_{j\in J} |\lambda_j|<\infty.
\]
Then $(\lambda_j \xx_j)_{j\in J}$ is a quasi-greedy basis for $\XX$.
\end{theorem}


A powerful tool to carry out estimates involving Jacobi polynomials is the so called \textit{Darboux formula}. The next theorem establishes an expression for the 
error term associated to this formula which is accurate enough for our purposes.

\begin{theorem}[cf.~\cite{Szego}*{Theorem~8.21.13}]
\label{result:Darboux}
Let $\alpha$, $\beta>-1$ and $\delta>0$. Then
\begin{equation*}
n^{1/2} P_n^{(\alpha,\beta)}(\cos \theta)
= k(\theta)  \cos(n\theta+\phi(\theta))+ E_n(\theta) 
\end{equation*}
with 
\begin{align*}
\phi(\theta)
&=(\alpha+\beta+1)\theta/2-(2 \alpha+1)\pi/4, \\
k(\theta)
&= \pi^{-1/2}(\sin \tfrac{\theta}{2})^{-\alpha-1/2}
(\cos \tfrac{\theta}{2})^{-\beta-1/2},
\end{align*}
and the error term $E_n(\theta)$ verifies
\[
E_n(\theta) = \frac{k(\theta)}{n \sin \theta} \, O(1)
\]
for $n\in\NN$, where the $O(1)$ holds uniformly in the interval 
${\delta}/{n}\le \theta \le \pi-{\delta}/{n}$.
\end{theorem}

Darboux formula provides tight estimates for Jacobi polynomials when the variable is not too close to the endpoints  $-1$ and $1$. The technique to estimate Jacobi polynomials near $1$ is also well-known for experts. It is based on the formula
\begin{equation}
\label{derivative}
\left(P_n^{(\alpha,\beta)}\right)'= (1+\alpha+\beta+n) P_{n-1}^{(\alpha+1,\beta+1)}
\end{equation}
and the behavior of the roots of Jacobi polynomials. In the following lemma we reproduce this standard argument for the sake of completeness.

\begin{lemma}
\label{result:NearOneEstimate}
Let $\alpha>-1$ and $\beta>-1$. There is $d>0$ such that
\[
P_n^{(\alpha,\beta)}(x)\approx n^\alpha
\]
for $n\in\NN$ and $1-d/n^2 \le x \le 1$.
\end{lemma}

\begin{proof}
Let  $z_n$ denote the largest root of $ P_n^{(\alpha,\beta)}$ and let $\gamma_n\in(0,\pi)$ be such that $\cos(\gamma_n)=z_n$.
It is known (see \cite{Szego}*{Theorem~8.9.1}) that $\gamma_n\approx 1/n$ and, consequently, $1-z_n\approx 1/n^2$. 
Moreover, it is easy to deduce from \eqref{ValueInOne} and \eqref{derivative} that 
\[
P_n^{(\alpha,\beta)}(1)\approx n^\alpha\quad\text{and}\quad(P_n^{(\alpha,\beta)})'(1)\approx n^{\alpha+2}.
\] 
Choosing $d>0$ small enough we get
$z_n \le 1-d/n^2$ and
\[
n^\alpha\lesssim P_n^{(\alpha,\beta)}(1) - \frac{d}{n^2} (P_n^{(\alpha,\beta)})'(1).
\]
Let $1-d/n^2 \le x \le 1$. 
Since $0\le (P_n^{(\alpha,\beta)})'(t)\le (P_n^{(\alpha,\beta)})'(1)$ for any $t\in[x,1]$,
\[
P_n^{(\alpha,\beta)}(x)
= P_n^{(\alpha,\beta)}(1) - \int_x^1 (P_n^{(\alpha,\beta)})'(t)\, dt 
\ge P_n^{(\alpha,\beta)}(1) - \frac{d}{n^2} (P_n^{(\alpha,\beta)})'(1).
\]
For the reverse inequality we note that $P_n^{(\alpha,\beta)}(x)\le P_n^{(\alpha,\beta)}(1)\approx n^\alpha$.
\end{proof}


Darboux formula allows us to compute the $L_p(\mu_{\alpha,\beta})$-norms of Jacobi polynomials.
Let $\alpha,\beta$ be such that $\min\{\alpha,\beta\}> -1/2$ and denote
\[
p(\alpha,\beta)=\frac{4(\gamma+1)}{2\gamma+3}, \quad
q(\alpha,\beta)=\frac{4(\gamma+1)}{2 \gamma +1}, \quad\text{ where }\gamma=\max\{\alpha,\beta\}.
\]
Notice that $p(\alpha,\beta)$ and $q(\alpha,\beta)$ are conjugate exponents.
 We have (cf.~\cite{Leves-Kus}) that, for $n\ge 2$,
\begin{equation}
\label{result:JacobiBasisNorm}
\begin{aligned}
 \Vert p_n^{(\alpha,\beta)} \Vert_{L_p(\mu_{\alpha,\beta})}
&\approx n^{1/2}\Vert P_n^{(\alpha,\beta)} \Vert_{L_p(\mu_{\alpha,\beta})} \\
&\approx
 \begin{cases} 1, & \text{ if } 1\le p<q(\alpha,\beta), \\
 (\log n)^{1/p}, & \text{ if } p=q(\alpha,\beta),\\
 n^{(2 \gamma+1)/2-2( \gamma+1)/p}, & \text{ if } q(\alpha,\beta)<p<\infty.
\end{cases}
\end{aligned}
\end{equation}

An elementary consequence of \eqref{result:JacobiBasisNorm} is that Jacobi polynomials,
when $\min\{\alpha, \beta\}<-1/2$, are not uniformly bounded.
Using the terminology of bases, \eqref{result:JacobiBasisNorm} yields the following result.

\begin{lemma}
\label{result:NormalizationAndBoundedness} 
Let $1\le p< \infty$ and  $\min\{\alpha, \beta\}>-1/2$.  Then
\begin{itemize}
\item[(a)] The Jacobi polynomials of indices $\alpha$ and $\beta$ form a bounded basis for $L_p(\mu_{\alpha,\beta})$ if an only if
$p(\alpha,\beta)<p<q(\alpha,\beta)$.
\item[(b)] If $p(\alpha,\beta)<p<q(\alpha,\beta)$, then both $(p_n^{(\alpha,\beta)})_{n=0}^\infty$ and $(n^{1/2} P_n^{(\alpha,\beta)})_{n=0}^\infty$
 form a semi-normalized bounded basis for $L_p(\mu_{\alpha,\beta})$.
\end{itemize}
\end{lemma}

\begin{remark}
Notice that the range of indices $p$ for which the Jacobi polynomials with $\min\{\alpha,\beta\}>-1/2$ form a bounded basis for $L_p(\mu_{\alpha,\beta})$ coincides with the range of indices for which they are a Schauder basis of $L_p(\mu_{\alpha,\beta})$ with the natural order (cf.~\cite{Pollard}).
\end{remark}

\begin{remark}
Lemma~\ref{result:NormalizationAndBoundedness}, combined with either Theorem~\ref{NotSimultaneouslyBounded} or~\cite{KadPel1962}*{Corollary~9}, gives that Jacobi polynomials with $\min\{\alpha,\beta\}>-1/2$ are not an unconditional basis for $L_p(\mu_{\alpha,\beta})$ unless $p=2$.
This result can also be obtained from \cite{KonigNielsen1994}*{Proposition~4}.
\end{remark}

Lemma~\ref{result:Random} provides a tool to check if a basis is a suitable candidate to be quasi-greedy, and leads us to compare norms of the form
$\Vert \sum_{j\in A} \xx_j \Vert_p$ with norms of the form $\Vert (\sum_{j\in A} |\xx_j|^2)^{1/2} \Vert_p$. 
In this direction, we state the following result.

\begin{proposition}
\label{result:AverageJacobiNorm} 
Let  $\alpha$, $\beta$ and $p$ be such that $\min\{\alpha, \beta\}>-1/2$ and
  $1\le p<q(\alpha,\beta)$.
Then, for $A\subseteq \NN$ finite,
\[
\left\Vert\left(\sum_{n\in A} \left(n^{1/2} P_n^{(\alpha,\beta)}\right)^2\right)^{1/2}\right\Vert_{L_p(\mu_{\alpha,\beta})} \approx 
\left\Vert\left(\sum_{n\in A} \left(p_n^{(\alpha,\beta)}\right)^2\right)^{1/2}\right\Vert_{L_p(\mu_{\alpha,\beta})} \approx |A|^{1/2}.
\]
\end{proposition}

\begin{proof}
Just combine Proposition~\ref{result:LdosSumEstimate} with~\eqref{result:JacobiBasisNorm}.
\end{proof}

Proposition~\ref{result:AverageJacobiNorm} says that the expected value of the $L_p(\mu_{\alpha,\beta})$-norms $\Vert \sum_{n\in A} \varepsilon_n n^{1/2} P_n^{(\alpha,\beta)}\Vert_{L_p(\mu_{\alpha,\beta})} $,
when $(\varepsilon_n)_{n\in A}$ runs over all possible signs $\{\pm 1\}^A$,
is of the order of $|A|^{1/2}$ (cf.\ \cite{AlbiacKalton2006}*{Theorem 6.2.13}). In next Proposition we find norms which deviate significantly from the average value for $p\not= 2$.

\begin{proposition}
\label{result:CCJacobiNorm} 
Let $\alpha$, $\beta$ and $p$ be such $\min\{\alpha, \beta\} >-1/2$ and 
$p(\alpha,\beta)<p<q(\alpha,\beta)$.
Then
\[
\left\Vert\sum_{n=0}^{N-1} (N+2n)^{1/2} P_{N+2n}^{(\alpha,\beta)}\right\Vert_{L_p(\mu_{\alpha,\beta})}\approx N^\omega, \quad N\in\NN,
\]
where $\omega=\max\{{(2\alpha+3)/2-2(\alpha+1)/p},{(2\beta+3)/2-2(\beta+1)/p}\}$.
\end{proposition}

\begin{proof}
 Since $P_n^{(\alpha,\beta)}(-x)=(-1)^n P_n^{(\beta,\alpha)}(x)$ it suffices to prove the estimate
\begin{equation*}
\label{eq:esti-alpha}
I_N:=\int_0^1\left|\sum_{n\in A_N} n^{1/2} P_{n}^{(\alpha,\beta)}(x)\right|^p \, d\mu_{\alpha,\beta}(x) \approx N^\sigma,
\end{equation*}
where $A_N=\{N+2n : 0\le n\le N-1\}$ and $\sigma=p(2\alpha+3)/2-2(\alpha+1)$. Notice that the hypothesis $p(\alpha,\beta)<p<q(\alpha,\beta)$ implies $0<\sigma<p$.

Consider $d>0$ as in Lemma~\ref{result:NearOneEstimate} and choose $\theta_N\in(0,\pi)$ such that
\[
\cos(\theta_N)=x_N=1-\frac{d}{(3N-2)^2}.
\]
Write $I_N=J_N+K_N$, where
\begin{align*}
J_N &= \int_{x_N}^1 \left|\sum_{n\in A_N} n^{1/2} P_{n}^{(\alpha,\beta)}(x)\right|^p \, d\mu_{\alpha,\beta}(x),\quad \text{and}\\
K_N &= \int_{0}^{x_N} \left|\sum_{n\in A_N} n^{1/2} P_{n}^{(\alpha,\beta)}(x) \right|^p \, d\mu_{\alpha,\beta}(x).
\end{align*}
Since $P_n(x)\approx n^\alpha$ for $N\in\NN$, $n\in A_N$ and $x\in[x_N,1]$,
\[
J_N\approx \left(\sum_{n\in A_N} n^{\alpha+1/2} \right)^p \int_{x_N}^1 (1-x)^\alpha\, dx \approx N^{p(\alpha+3/2)} N^{-2(\alpha +1)} = N^{\sigma}.
\]
Let $k$, $\phi$, and $E_n$ be as is Theorem~\ref{result:Darboux}.
We have $K_N \le 2^{p-1} (L_N+ M_N)$, where
\begin{align*}
L_N &= \int_0^{x_N}\left| \sum_{n\in A_N} E_{n}(\arccos x) \right|^p \, d\mu_{\alpha,\beta}(x),\quad\text{and}\\
M_N &= \int_0^{x_N}\left| k(\arccos x) \sum_{n\in A_N} \cos(n \arccos x+\phi(\arccos x) )\right|^p \, d\mu_{\alpha,\beta}(x).
\end{align*}

Since $\theta_N\approx N^{-1}$, there exists $\delta>0$ such that $\delta/N\le \theta_N$. By Theorem~\ref{result:Darboux},
\[
|E_n(\theta)|\lesssim \frac{1}{n} (\sin \tfrac{\theta}{2})^{-\alpha-3/2}
\]
for $N\in\NN$, $n\in A_N$ and $\theta\in[\theta_N,\pi/2]$.
The change of variable $x=\cos\theta$ yields
\[
L_N\lesssim\left( \sum_{n\in A_N} \frac{1}{n}\right)^p \int_{\theta_N}^{\pi/2} ( \sin\tfrac{\theta}{2})^ {-1-\sigma} \, d\theta
\approx \int_{\theta_N}^\infty \theta^{-1-\sigma}\, d\theta = \frac{1}{\sigma} \theta_N^{-\sigma}\approx N^\sigma.
\]
Using the change of variable $x=\cos\theta$ and the formula
\[
\left| \sum_{n\in A_N} \cos(n\theta+\phi(\theta))\right|
= \frac{|\sin(N\theta) \cos(2N\theta+\phi(\theta))|}{\sin \theta},
\]
which is obtained taking into account that we are adding the real part of a geometric sum, gives
\begin{align*}
M_N &\approx \int_{0}^{\pi/2} |\sin(N\theta) \cos(2N\theta+\phi(\theta))|^p (\sin\tfrac{\theta}{2})^{-1-\sigma}\,d\theta \\
&\approx \int_{0}^{\pi/2} |\sin(N\theta) \cos(2N\theta+\phi(\theta))|^p \theta^{-1-\sigma}\,d\theta \\
&= N^\sigma \int_{0}^{N\pi/2} |\sin(u) \cos(2u+\phi(u/N))|^p u^{-1-\sigma} \, du \\
& \approx N^\sigma.
\end{align*}
To deduce the last step in the estimate of $M_{N}$ we have used that
\[
 \int_0^\infty |\sin(u)|^p  u^{-1-\sigma} \, du<\infty
\]
and the Dominated Convergence Theorem.
\end{proof}

We are now in a position to complete the proof Theorem~\ref{result:MainTheorem} as advertised.

\begin{proof}[Proof of Theorem~\ref{result:MainTheorem}]
Assume that the $L_p(\mu_{\alpha,\beta})$-normalized sequence of Jacobi polynomials 
is a quasi-greedy basis for $L_p(\mu_{\alpha,\beta})$.
Then, thanks to Lemma~\ref{result:NormalizationAndBoundedness}(a), $p(\alpha,\beta)<p<q(\alpha,\beta)$.

By Theorem~\ref{result:Perturbing} and Lemma~\ref{result:NormalizationAndBoundedness}(b), the basis
$(n^{1/2} P_n^{(\alpha,\beta)})_{n=0}^\infty$ is also quasi-greedy for $L_p(\mu_{\alpha,\beta})$. Combining Proposition~\ref{result:Democracy} 
(or Lemma~\ref{result:Random} together with Proposition~\ref{result:AverageJacobiNorm}) with
Proposition~\ref{result:CCJacobiNorm} we obtain
$N^\omega \approx N^{1/2}$ for $N\in\NN$.
Therefore, $\omega=1/2$, i.e.~$p=2$.
\end{proof}

\subsection*{Acknowledgements}
The first two authors were partially supported by the Spanish Research Grant \textit{An\'alisis Vectorial, Multilineal y Aplicaciones}, reference number MTM2014-53009-P, and the last two authors were partially supported by the Spanish Research Grant \textit{Ortogonalidad, Teor\'{\i}a de la Aproximaci\'on y Aplicaciones}, reference number MTM2012-36732-C03-02. The first-named author  also acknowledges the support of Spanish Research Grant \textit{ Operators, lattices, and structure of Banach spaces}, 
with reference MTM2012-31286.



\end{document}